\documentclass[a4paper, 10pt]{article}
\usepackage{mathpaper}
\usepackage{mathshort}
\usepackage{tgpagella}
\usepackage[T1]{fontenc}
\usepackage[final]{showlabels}

\newcommand{\aks}{\mathcal{KS}_{\mathrm{ad}}}

\title{On Cowling's $\rmL^p$-integrability conjecture for Kunze--Stein groups}
\author{Siwei Liang}
\date{\today}

\begin{document}
    \maketitle
    \begin{abstract}
        We prove a stronger form of Cowling's $\rmL^p$-integrability conjecture for an admissible class of Kunze--Stein groups containing all $S$-algebraic groups and all closed CCR boundary-transitive subgroups of tree automorphism groups. We also disprove a cyclic version formulated by Samei--Wiersma.
    \end{abstract}

\section{Introduction}
    In his 1978 paper \cite{cowling1978the-kunze-stein}, M.~Cowling established the Kunze--Stein phenomenon for every connected semisimple Lie group with finite center: for $1\leq r<2$, convolution extends to a bounded bilinear map $\rmL^r(G) \times \rmL^2(G) \to \rmL^2(G)$, generalizing the earlier work of Kunze--Stein \cite{kunze-stein1960uniformly} on $\SL(2,\bbR)$.
    More generally, a \emph{Kunze--Stein (KS) group} is a unimodular locally compact group $G$ satisfying this property.
    In the same paper, Cowling formulated the following conjecture.
    \begin{conjecture}[{\cite[p.~233]{cowling1978the-kunze-stein}}]\label{conj cowling: Lp coeff KS groups}
        Let $G$ be a Kunze--Stein group and $\pi$ be an irreducible unitary representation of $G$. If a nonzero matrix coefficient of $\pi$ lies in $\Lspace^p(G)$ for some $p\in (2,\infty)$, then all matrix coefficients of $\pi$ lie in $\rmL^p(G)$ as well.
    \end{conjecture}

    More recently, Samei--Wiersma \cite{samei-wiersma2024exotic} obtained a near solution to this conjecture by proving the statement with $\Lspace^p$ replaced by $\Lspace^{p+}$ (see \autoref{thm samei-wiersma: almost Lp implies uniformly Lp, kunze-stein unimodular}). However, they formulated and attributed to Cowling the following cyclic version of \autoref{conj cowling: Lp coeff KS groups}.
    \begin{conjecture}[{\cite[Conj 1.4]{samei-wiersma2024exotic}}]\label{conj samei-wiersma: cyclic}
        Let $G$ be a Kunze--Stein group and $(\pi, \scrV)$ be a unitary representation of $G$ with a cyclic vector $\bfv$ such that $\Innerprod{\pi(\cdot)\bfv,\bfv}\in \Lspace^p(G)$ for some $p\in (2,\infty)$, then all matrix coefficients of $\pi$ lie in $\Lspace^p(G)$ as well.
    \end{conjecture}

    \begin{center}
        \emph{The purpose of the present paper is to prove a strengthening of \autoref{conj cowling: Lp coeff KS groups} for the major known classes of KS groups and provide a counterexample to \autoref{conj samei-wiersma: cyclic}.}
    \end{center}
    
    To the best of our knowledge, major known classes of KS groups include:
    \begin{enumerate}[label=(\roman*)]
        \item connected semisimple Lie groups with finite center (Cowling \cite{cowling1978the-kunze-stein});\label{ks1}
        \item groups of $\bbF$-rational points of simply connected simple algebraic groups over a nondiscrete totally disconnected local field $\bbF$ (Veca \cite{veca2002the-kunze-stein});\label{ks2}
        \item closed boundary-transitive subgroups of automorphism groups of thick (semi-)homogeneous trees (Nebbia \cite{nebbia1988groups}).\label{ks3all}
    \end{enumerate}
    Recall that a locally compact second countable (lcsc) group $G$ is \emph{CCR} or \emph{liminal} if for every irreducible unitary representation $\pi$ and every $\phi\in\Cc(G)$, the operator $\pi(\phi)$ is compact, \emph{cf.} \cite[Def 6.E.7]{bekka-harpe2020unitary}.
    In this paper, we will not directly address the class \ref{ks3all} but its subclass:
    \begin{enumerate}
        \makeatletter
        \item[(iii')] \def\@currentlabel{(iii')}\label{ks3} closed CCR boundary-transitive subgroups of automorphism groups of thick (semi-)homogeneous trees.
        \makeatother
    \end{enumerate}
    The equality of \ref{ks3all} and \ref{ks3} amounts to Nebbia's CCR conjecture \cite{nebbia1999groups} which is open in general. It is however known that the class \ref{ks3} contains:
    \begin{itemize}
        \item all those in \ref{ks3all} with Tits' independence property (Amann \cite{amann2003groups});
        \item Radu groups in the sense of Semal \cite{semal2024radu}.
    \end{itemize}

    To formulate our theorem, let us consider the following class.
    \begin{definition}
        The \emph{admissible Kunze--Stein class} $\aks$ is defined to be the smallest class of lcsc groups which contains all compact groups and the groups in \ref{ks1}, \ref{ks2}, \ref{ks3} and is closed under the elementary operations:
        \begin{enumerate}[label=(\alph*)]
        \item finite direct products,\label{a}
        \item closed finite-index subgroups and overgroups,\label{b}
        \item quotients by compact normal subgroups.\label{c}
    \end{enumerate}
    \end{definition}
    \begin{remark*}
        The class $\aks$ consists of Kunze--Stein groups by \cite[Lem 7.1]{cowling1978the-kunze-stein}. It contains all the $S$-algebraic groups in the sense of Gorodnik--Nevo \cite[Def 3.4]{gorodnik-nevo2009the-ergodic}.
    \end{remark*}

    Our main result is the following strengthening of \autoref{conj cowling: Lp coeff KS groups} for $\aks$.
    \begin{theorem}\label{mythm: S-algebraic, Lp irreducible ur}
        Let $G\in\aks$ and $\pi$ be an irreducible unitary representation of $G$. If a nonzero matrix coefficient of $\pi$ lies in $\Lspace^p(G)$ for some $p\in (2,\infty)$, then there exists some $q<p$ such that all matrix coefficients of $\pi$ lie in $\Lspace^{q}(G)$.
    \end{theorem}

    Recall that a unitary representation of a Kunze--Stein group $G$ is tempered iff its matrix coefficients lie in $\rmL^{2+
    \varepsilon}(G)$ for all $\varepsilon>0$. We have the following consequence for the non-tempered representations in the unitary dual $\wh{G}$.

    \begin{corollary}
        For $G\in\aks$ and a non-tempered $\pi\in\wh{G}$, there is a unique number $p(\pi)\in (2,+\infty]$ such that 
        \begin{itemize}
            \item if $p(\pi)<+\infty$, then for every nonzero matrix coefficient $c$ of $\pi$, one has $c\in\rmL^q(G)$ iff $q>p(\pi)$;
            \item otherwise if $p(\pi)=+\infty$, then $c\notin\rmL^q(G)$ for every finite $q$.
        \end{itemize}
    \end{corollary}

    The irreducibility assumption in \autoref{conj cowling: Lp coeff KS groups} is indispensable: in \autoref{prop: disprove samei-wiersma conj, SL2R}, we prove that \autoref{conj samei-wiersma: cyclic} fails even for $G=\SL(2,\bbR)$. The construction therein extends readily to other groups.

    The proof of \autoref{mythm: S-algebraic, Lp irreducible ur} is a combination of $(1^\circ)$ an elementary reduction to $K$-finite coefficients, $(2^\circ)$ the theory of leading exponents and $(3^\circ)$ the Samei--Wiersma theorem. 
    One novelty in our proof is an adaptation of the theory of leading exponents from semisimple groups to the tree automorphism groups.
    The logic of the proof can be summarized by:
    \begin{align*}
        \text{one coefficient is $\rmL^p$} &\overset{(1^\circ)}{\implies} \text{all $K$-finite coefficients are $\rmL^p$} \\
        & \overset{(2^\circ)}{\implies}
        \text{all $K$-finite coefficients are $\rmL^{q_0}$ for some $q_0<p$}\\
        & \overset{(3^\circ)}{\implies}\text{all coefficients are $\rmL^{q_0+}$; conclude for $q\in(q_0,p)$.}
    \end{align*}

    \subsection*{Acknowledgements and declaration of AI use}
    The author thanks his advisor Yves Benoist for his guidance, helpful discussions on the counterexample, and comments on an earlier draft.

    The author used ChatGPT for proofreading and to improve the phrasing and the organization of the manuscript.
    The final text was reviewed and edited by the author, who takes full responsibility for the content.

\section{Integrability of matrix coefficients}\label{sec: int mat coeff}
    For the representation theory of lcsc groups, we refer to \cite{bekka-harpe2020unitary}.

    Let us recall different notions for the integrability of matrix coefficients.
    \begin{definition}
        Given $p\in [1,\infty)$, a unitary representation $(\pi, \scrV)$ of an lcsc group $G$ is said to be
        \begin{enumerate}
            \item \emph{strongly $L^p$} if there exists a dense subspace $\scrV_0\subset \scrV$ such that $\Innerprod{\pi(\cdot)\bfv,\bfv}\in \Lspace^p(G)$ for all $\bfv\in\scrV_0$;
            \item \emph{strongly $L^{p+}$} (or \emph{almost $L^{p}$}) if there exists a dense subspace $\scrV_0\subset\scrV$ such that $\Innerprod{\pi(\cdot)\bfv,\bfv}\in \Lspace^{p+\varepsilon}(G)$ for all $\bfv\in\scrV_0$ and $\varepsilon>0$;
            \item \emph{uniformly $L^p$} if $\Innerprod{\pi(\cdot)\bfv,\bfv}\in\Lspace^p(G)$ for all $\bfv\in\scrV$;
            \item \emph{uniformly $L^{p+}$} if it is uniformly $\Lspace^{p+\varepsilon}$ for all $\varepsilon>0$.
        \end{enumerate}
    \end{definition}
    \begin{remark*}
        The terminology of these notions is not entirely consistent in the literature: the same term may be given different meanings, while the same notion may appear under different names.
    \end{remark*}

    Uniformly $\Lspace^p$ representations admit the following characterizations.
    \begin{lemma}[{\cite[Lem 27]{kunze-stein1960uniformly}, \cite[Lem 1.1]{cowling1978the-kunze-stein}}]\label{lem kunze-stein: uniformly Lp iff op norm bounded by Lp' norm}
        Let $p, p'\in(1,\infty)$ be with $1/p+1/p'=1$. For a unitary representation $(\pi,\scrV)$ of an unimodular lcsc group $G$, TFAE:
        \begin{enumerate}
            \item $\pi$ is uniformly $\Lspace^p$;
            \item there is a constant $C\in \bbR^+$ with $\opnorm{\pi(\phi)} \leq C \Norm{\phi}_{p'}$ for all $\phi\in\Cc(G)$;
            \item there is a constant $C\in\bbR^+$ with $\Norm{\inn{\pi(\cdot)\bfv, \bfw}}_{p}\leq C$ for all unit $\bfv,\bfw\in\scrV$;
            \item $A_\pi\subset \rmL^p(G)$ in the notation of \cite{cowling1978the-kunze-stein,samei-wiersma2024exotic}, where $A_{\pi}$ denotes the image in $\rmL^{\infty}(G)$ of the projective tensor product $\scrV\otimes_{\mathrm{pr}}\overline{\scrV}$ under the continuous bilinear map $\bfv\otimes \overline{\bfw}\to\inn{\pi(\cdot)\bfv,\bfw}$. \qed
        \end{enumerate}
    \end{lemma}

    Given two unitary representations $\pi, \sigma$ of an lcsc group $G$, we say that $\sigma$ is \emph{weakly contained} in $\pi$ if $\opnorm{\sigma(\phi)} \leq \opnorm{\pi(\phi)}$ for all $\phi\in \Cc(G)$. For equivalent definitions, see \cite[\S 1.C]{bekka-harpe2020unitary}.
    A consequence of \autoref{lem kunze-stein: uniformly Lp iff op norm bounded by Lp' norm} is that the uniform $\Lspace^p$ or uniform $\Lspace^{p+}$ property passes to all weakly contained representations.
    \begin{lemma}\label{lem: uniformly Lp passes to weakly contained}
        Let $G$ be an lcsc group and $\pi$ be a unitary representation of $G$. 
        If $\pi$ is uniformly $\Lspace^{p}$ (resp. uniformly $\rmL^{p+}$), then so is any $\sigma$ weakly contained in $\pi$.\qed
    \end{lemma}

    Finally, we recall the Samei--Wiersma theorem, rephrased in our language.
    \begin{theorem}[{\cite[Thm 5.3]{samei-wiersma2024exotic}}]\label{thm samei-wiersma: almost Lp implies uniformly Lp, kunze-stein unimodular}
        Let $G$ be a Kunze--Stein group and $p\geq 2$. Then every strongly $\Lspace^{p+}$ unitary representation of $G$ is uniformly $\Lspace^{p+}$.\qed
    \end{theorem}

\section{The admissible Kunze--Stein class}\label{sec: ad KS class}
\subsection{The Hecke algebra}
We introduce a basic tool for the groups in $\aks$.
Let $G\in\aks$ and $K<G$ be a compact subgroup of $G$. Recall that a function is \emph{$K$-bi-finite} if its left and right $K$-translates span a finite-dimensional space.
The \emph{Hecke algebra of $G$ with respect to $K$} is defined by
\begin{equation}\label{eq def: H_K, ss Lie}
    \calH_K(G) \coloneq \setdef{\phi\in \Cc(G) : \textup{$\phi$ is $K$-bi-finite}}.
\end{equation}
Under convolution, $\calH_K(G)$ forms a complex algebra.

If one replaces $K$ by a commensurable $K'<G$, then $K$-finiteness and $K'$-finiteness are equivalent and hence $\calH_K(G)=\calH_{K'}(G)$. If one replaces $K$ by a conjugate, then the associated Hecke algebras are conjugate as well.

For primitive groups, choosing $K$ to be a maximal compact subgroup poses no essential ambiguity.
Maximal compact subgroups of connected semisimple Lie groups are conjugate. In the non-Archimedean and the tree cases, the relevant groups are totally disconnected, for which the identity element admits a neighborhood base consisting of compact open subgroups (van Dantzig's theorem). The maximal compact subgroups are open and hence commensurable.

To choose a compact subgroup $K$ of finite index in a maximal compact subgroup for each $G\in\aks$, consider the elementary operations.
If $G=G_1\times G_2$ with $K_i < G_i$ chosen, then choose $K=K_1\times K_2$. 
If $G<G_0$ is a finite-index subgroup with $K_0<G_0$ chosen, then choose $K=G\cap K_0$ which has finite index in $K_0$.
If $G>G_0$ is a finite-index overgroup with $K_0<G_0$ chosen, then choose $K=K_0$.
If $G=G_0/C$ is a quotient of $G_0$ by a compact normal subgroup $C$ with $K_0<G_0$ chosen, then choose $K$ to be the image of $K_0C$.

For $G\in\aks$, define the \emph{Hecke algebra of $G$} to be $\calH(G)\coloneq \calH_K(G)$ with respect to the chosen $K<G$. Then $\calH(G)$ is uniquely defined up to conjugation. In any case, it is $\rmL^1$-dense in $\Cc(G)$.

\subsection{Admissibility and CCR}
Given a unitary representation $(\pi,\scrV)$ of $G$, we say a vector $\bfv\in \scrV$ is $K$-finite, if $\pi(K)\bfv$ span a finite-dimensional subspace. Denote by $\scrV_K$ the space of $K$-finite vectors in $\scrV$. 
Every irreducible representation $\tau$ of $K$ is finite-dimensional.
Denote by $\scrV(\tau)$ the subspace of vectors transforming by $\tau$ under $K$. Then $\scrV_K$ is the algebraic direct sum of $\scrV(\tau)$ over $\tau\in\wh{K}$.
By Peter--Weyl, $\scrV_K$ is dense in $\scrV$.

We say $\pi$ is \emph{$K$-admissible}, if for every $\tau\in\wh{K}$, the subspace $\scrV(\tau)$ has finite dimension, or equivalently the multiplicity of $\tau$ in $\scrV$ is finite.
When $G$ is totally disconnected, the admissibility is equivalent to saying that the subspace $\scrV^U$ of $U$-invariant vectors has finite dimension for every compact open $U<K$.

Note that replacing $K$ by a conjugate or commensurable subgroup does not affect the admissibility of $\pi$. 
Hence for any $G\in\aks$, the admissibility of representations is independent of the choice of $K$.

The main result of this section is the following admissibility statement. In the literature, one also says that $K$ is large in $G$, \emph{cf.} \cite[Def 6.E.10]{bekka-harpe2020unitary}.
\begin{lemma}\label{lem: irreducible implies admissible}
    For $G\in\aks$, every irreducible unitary representation is admissible.
    Every $G\in\aks$ is CCR.
\end{lemma}

It is well known that the admissibility statement implies the CCR property of lcsc groups, \emph{cf.} \cite[Prop 6.E.11]{bekka-harpe2020unitary}; the converse holds for the totally disconnected ones, \emph{cf.} \cite[Prop 2.14]{gorfine2026a-spectral}. Hence, the class \ref{ks3} satisfies \autoref{lem: irreducible implies admissible}.

For the other primitive groups, admissibility is classical.
For semisimple Lie groups, this is due to Harish-Chandra \cite{harish-chandra1953representations}. For non-Archimedean semisimple groups, this is due to Bernstein \cite{bernshtein1974all-reductive}. 
For compact groups, this is trivial.
Consequently, all primitive groups are CCR and hence type-I.

\begin{lemma}
    Every $G\in \aks$ is of type I.
\end{lemma}
\begin{proof}
    It is well known that the elementary operations preserve the type-I property: see \cite[13.11.7]{dixmier1977c-algebras} for \ref{a} and \cite[4.3.5]{dixmier1977c-algebras} for \ref{b}, \ref{c}.
\end{proof}

\begin{proof}[Proof of \autoref{lem: irreducible implies admissible}]
    It suffices to prove that the admissibility statement is preserved under the elementary operations \ref{a}, \ref{b}, \ref{c}.

    (a) Let $G=G_1\times G_2$ with both $G_i$ satisfying the lemma. By the type-I property, every irreducible unitary representation $\pi$ of $G$ satisfies $\pi\cong \pi_1\boxtimes \pi_2$ for some $\pi_i\in\wh{G}_i$. Since $\pi_i$ is $K_i$-admissible for each $i$, we deduce that $\pi$ is $(K_1\times K_2)$-admissible. The lemma holds for $G$.

    (b1) Let $G<G_0$ be a closed finite-index subgroup with $G_0$ satisfying the lemma. Given $\pi\in\wh{G}$, consider the induced representation $\pi_0\coloneq \Ind_G^{G_0}\pi$. As $[G_0:G]<\infty$, Mackey's theory implies that $\pi_0$ splits into a finite direct sum of irreducible unitary representations. Thus $\pi_0$ is $K_0$-admissible. Since $\pi\subset \pi_0|_G$ and $[K_0 : K_0\cap G]<\infty$, we deduce that $\pi$ is $(K_0\cap G)$-admissible.

    (b2) Let $G>G_0$ be a closed finite-index overgroup with $G_0$ satisfying the lemma. We may assume $G_0\lhd G$; otherwise pass $G_0$ to its normal core and apply (b1). Given $\pi\in\wh{G}$, Mackey's theorem implies that the restriction $\pi|_{G_0}$ is a finite direct sum of irreducible unitary representations of $G_0$ and hence is $K_0$-admissible. As $[K:K_0]<\infty$, we deduce that $\pi$ is $K$-admissible as well.

    (c) Let $G=G_0/C$ be a quotient by a compact normal subgroup $C$ with $G_0$ satisfying the lemma. The irreducible unitary representations of $G$ are naturally identified with $C$-trivial representations of $G_0$, for which $K$-admissibility and $K_0$-admissibility are equivalent.
\end{proof}

\subsection{$K$-finite vectors and the Hecke algebra}
Recall that for a unitary representation $(\pi,\scrV)$ and $\phi\in\Cc(G)$, the operator $\pi(\phi)$ is defined by
\begin{equation*}
    \pi(\phi)\bfv = \int_G \phi(g)\pi(g)\bfv \dif g.
\end{equation*}
Since $\pi(g)\pi(\phi) = \pi(L_g \phi)$, the operators $\pi(\calH(G))$ always range in $\scrV_K$.

The following lemma illustrates the crucial role of the Hecke algebra.

\begin{lemma}\label{lem: H(G) gives all K-finite}
    Let $(\pi,\scrV)$ be an irreducible unitary representation of $G\in\aks$. Then for any nonzero vector $\bfv\in\scrV$, we have $\pi(\calH(G))\bfv = \scrV_K$.
\end{lemma}
\begin{proof}
    By the irreducibility of $\pi$, the subspace $\pi(\Cc(G))\bfv$ is dense in $\scrV$. Since $\calH(G)$ is dense in $\Cc(G)$, we further deduce that $\pi(\calH(G))\bfv$ is dense in $\scrV$.
    
    Given any $\tau\in\wh{K}$, consider the associated idempotent 
    $ e_\tau(k) \coloneq d_\tau \overline{\tr\tau(k)}\dif k$ where $\dif k$ is the normalized Haar measure on $K$, so that $P_\tau\coloneq \pi(e_\tau)$ is the orthogonal projection onto the subspace $\scrV(\tau)$. Note that $e_\tau * \calH(G)\subset \calH(G)$.
    The density of $\calH(G)\subset \Cc(G)$ implies the density of
    $$ \pi(e_\tau * \calH(G))\bfv = P_{\tau}\pi(\calH(G))\bfv \subset P_\tau(\scrV) = \scrV(\tau).$$
    But by admissibility (\autoref{lem: irreducible implies admissible}), $\scrV(\tau)$ is finite-dimensional, whence
    \begin{equation*}
        \pi(e_\tau * \calH(G))\bfv = \scrV(\tau).
    \end{equation*}
    We further deduce that
    \begin{equation*}
        \scrV_K = \bigoplus_{\tau} \pi(e_\tau * \calH(G))\bfv \subset \pi(\calH(G))\bfv.
    \end{equation*}
    But since clearly $\pi(\calH(G))\subset\scrV_K$, the inclusion above must be an equality.
\end{proof}

\section{Proof of the main theorem}\label{sec: proof}

\subsection{Reduction to $K$-finite coefficients}
Using the Hecke algebra, we can promote the integrability of an arbitrary matrix coefficient to all the $K$-finite ones.
\begin{lemma}\label{lem: reduction to k-finite}
    Let $G\in\aks$ and $(\pi,\scrV)$ be an irreducible unitary representation of $G$. If a nonzero matrix coefficient of $\pi$ lies in $\rmL^p(G)$ for some $p\geq 1$, then  we have $\inn{\pi(\cdot)\bfv,\bfw}\in\rmL^p(G)$ for all $\bfv,\bfw\in\scrV_K$.
\end{lemma}
\begin{proof}
    Assume that for nonzero $\bfv_0,\bfw_0\in\scrV$, we have $\inn{\pi(\cdot)\bfv_0,\bfw_0}\in\rmL^p(G)$.
    Given any $\bfv,\bfw\in\scrV_K$, we deduce from \autoref{lem: H(G) gives all K-finite} the existence of $\phi,\psi\in\calH(G)$ with $\pi(\phi)\bfv_0=\bfv$ and $\pi(\psi)\bfw_0=\bfw$, whence
    \begin{equation*}
        \inn{\pi(\cdot)\bfv,\bfw} = \overline{\psi} * \inn{\pi(\cdot)\bfv_0, \bfw_0} * \check{\phi}.
    \end{equation*}
    Since $\calH(G)\subset \rmL^1(G)$, Young's inequality yields $\inn{\pi(\cdot)\bfv,\bfw}\in\rmL^p(G)$.
\end{proof}

\subsection{Openness of the integrability exponent}
Next, we show that for the primitive groups, the $\rmL^p$-integrability of $K$-finite matrix coefficients of an irreducible $\pi$ is an open condition in $p$.
The underlying reason is that their asymptotics are governed by finitely many exponential-polynomial terms.
We will treat the semisimple and the tree cases separately.

\begin{lemma}\label{lem: open condition in p, Lp of K-finite coefficients}
    Let $G$ be in one of the classes \ref{ks1}, \ref{ks2}, \ref{ks3}. Given $\pi\in\wh{G}$ and $p\in (1,\infty)$, if $\Innerprod{\pi(\cdot)\bfv,\bfw}\in \Lspace^p(G)$ for all $\bfv,\bfw\in \scrV_K$, then there exists some $q < p$ such that $\Innerprod{\pi(\cdot)\bfv,\bfw}\in \Lspace^q(G)$ for all $\bfv,\bfw\in \scrV_K$.
\end{lemma}

\subsubsection*{Semisimple case}
For the classes \ref{ks1}, \ref{ks2}, \autoref{lem: open condition in p, Lp of K-finite coefficients} is a direct consequence of the theory of leading exponents/characters, which we use below as a black box.
In the real case, the theory goes back to Harish-Chandra \cite{harish-chandra1959some,harish-chandra1984some,harish-chandra1984differential} and was reformulated by Casselman--Mili\v{c}i\'{c} in \cite{casselman-milicic1982asymptotic}. 
The theory in the non-Archimedean case was initiated by Harish-Chandra \cite{harish-chandra1970harmonic} and generalized by Silberger \cite{silberger1982asymptotics}. 
The following statement is derived from  \cite[Thm 8.14]{casselman-milicic1982asymptotic} (real) and \cite[Prop 2.5]{silberger1982asymptotics} (non-Archimedean).

\begin{lemma}\label{lem: K-finite coeff, from p to q<p, semisimple}
    For an irreducible admissible representation of a semisimple algebraic group over a nondiscrete local field, the $\rmL^p$-integrability of all K-finite matrix coefficients is characterized by finitely many strict linear inequalities involving $p$ and the leading exponents, therefore remaining valid for every $q<p$ sufficiently close to $p$.\qed
\end{lemma}

\subsubsection*{Tree case}
For the generalities on groups in class \ref{ks3}, we refer to \cite{amann2003groups,semal2023irreducibly}.
For certain $G$, it is known that every irreducible non-tempered unitary representation of $G$ is spherical, \emph{cf.} \cite{amann2003groups} for $G$ with Tits' independence property and \cite{semal2024radu} for Radu groups; in this case, the proof of \autoref{lem: open condition in p, Lp of K-finite coefficients} (for non-tempered $\pi$) is a simple computation of spherical functions, \emph{cf.} \cite[Lem 4.11]{heinig-laat2024group}. However, this property fails in general. We bypass this issue by adapting the theory of leading exponents to the tree setting.

Following \cite[\S4]{heinig-laat2024group}, let $T$ be a semi-homogeneous tree of degree $(d_0,d_1)$ and $G<\Aut(T)$ be a noncompact closed subgroup which acts transitively on the boundary at infinity of $T$. 
Assume $T$ to be thick, \emph{i.e.} $d_0,d_1\geq 3$.
Then the number $\kappa$ of $G$-orbits in $V(T)$ always lies in $\setdef{1,2}$, \emph{cf.} \cite[Prop 3.1]{heinig-laat2024group}.

Fix a hyperbolic element $a\in G$ of translation length $\kappa$, a vertex $o\in V(T)$ on the axis of $a$ and $K\coloneq\Fix_G(o)$. Then the $G$-orbit of $o$ can be seen as $Go = \setdef{x\in V(T): d(x,o)\in\kappa\bbN}$ and $K\backslash G / K = \setdef{Ka^n K: n\in \bbN}$.

Let $\scrW\subset \scrV_K$ be a $K$-irreducible subspace of type $\tau\in\wh{K}$ and $P_{\scrW}\colon \scrV\to \scrW$ be the orthogonal projection onto $\scrW$.
In the spirit of \cite[Ch VIII]{knapp1986representation}, we consider the vector-valued ``spherical function''
\begin{equation*}
    \Phi\colon G\to \End_{\bbC}(\scrW),\quad \Phi(g) \coloneq P_{\scrW} \pi(g) P_{\scrW}.
\end{equation*}
For any $\bfw_1,\bfw_2\in\scrW$, we have $\inn{\pi(g)\bfw_1,\bfw_2} = \inn{\Phi(g)\bfw_1,\bfw_2}$.
For any $k_1,k_2\in K$, we have $\Phi(k_1 a k_2) = \tau(k_1)\Phi(a)\tau(k_2)$.
Since $K$ is totally disconnected, $\ker\tau$ has finite index in $K$.
In particular, $\Phi$ is bi-invariant for the action of $\ker\tau$.

\begin{lemma}[{\cite[Lem 1]{willis1994the-structure}}]\label{lem: tidy subgroup U}
    There is a compact open subgroup $U<G$ contained in $\ker\tau$ which satisfies $U=U_+ U_- = U_- U_+$ for
    \begin{equation*}
        U_+\coloneq \bigcap_{n\geq 0} a^n U a^{-n}, \qquad U_- \coloneq \bigcap_{n\geq 0} a^{-n} U a^n.\qedineq
    \end{equation*}
\end{lemma}

\begin{lemma}\label{lem: UaU nth power = Ua^n U}
    $(UaU)^n = Ua^n U$ for all $n\in\bbN^*$.
\end{lemma}
\begin{proof}
    It suffices to prove that $(UaU)^n\subset U a^n U$. Since $a^{-1}U_+ a \subset U_+$, we have $UaU = U_{-} U_{+} a U = U_{-} a U$. Since $aU_{-}a^{-1}\subset U_{-}$, we deduce
    \begin{equation*}
        (Ua^nU)(UaU) = Ua^nU_{-} a U \subset Ua^{n+1} U.
    \end{equation*}
    By induction on $n$, we deduce $(UaU)^n\subset U a^n U$ for all $n$.
\end{proof}

Fix a Haar measure $\mu$ on $G$. For $n\in\bbN$, denote 
    $$\mu_n \coloneq \frac{1}{\mu(Ua^n U)}\indicator_{Ua^n U}\in\calH(G).$$
\begin{lemma}\label{lem: mu1 *n = mun}
    $\mu_1^{*n} = \mu_n$ for all $n\in\bbN^*$.
\end{lemma}
\begin{proof}
    Since $\mu_1$ is a $U$-bi-invariant probability measure, so is the convolution $\mu_1^{*n}$. But by \autoref{lem: UaU nth power = Ua^n U}, $\mu_1^{*n}$ is supported on the double coset $Ua^n U$, whence one must have $\mu_1^{*n} = \mu_n$.
\end{proof}

Consider the Hecke operator $\matr{A}\coloneq \pi(\mu_1)|_{\scrV^U}$ lying in $\End_{\bbC}(\scrV^U)$. \autoref{lem: mu1 *n = mun} gives $\pi(\mu_n)|_{\scrV^U} = \matr{A}^n$.
Since $\scrW\subset \scrV^U$ and $\Phi(a^n)=P_{\scrW}\pi(\mu_n) P_{\scrW}$, we have
\begin{equation*}
    \Phi(k a^n l) = \tau(k)P_{\scrW} \matr{A}^n P_{\scrW} \tau(l), \quad \text{for all $k,l\in K$, $n\in\bbN$.}
\end{equation*}
Let $r\coloneq [K:U]<\infty$ and $K = \sqcup_{i=1}^{r} Uk_i = \sqcup_{j=1}^{r} l_j U$, so that
\begin{equation}\label{eq intm: double U-coset}
    U\backslash G / U = \setdef{U k_i a^n l_j U : 1\leq i,j\leq r,\, n\in\bbN}.
\end{equation}
Note that different parameters on the RHS may give the same $U$-double coset, but the multiplicities are always bounded above by $r^2$.
As $\dim \scrV^{U}<\infty$ by admissibility, we can apply the Jordan normal form of $\matr{A}$ to obtain the expansion
\begin{equation}\label{eq intm: expansion of Phi}
    \Phi(k_i a^n l_j) = \sum_{\lambda\in\sigma(\matr{A})} \sum_{s=0}^{m_{\lambda}-1} n^s \lambda^n \matr{C}_{\lambda,s,i,j},
\end{equation}
where $\matr{C}_{\lambda,s,i,j}\in \End_{\bbC}(\scrW)$ does not depend on $n$.
Let $R$ be the largest modulus of those $\lambda\in\sigma(\matr{A})$ occurring on the RHS. We fix vectors $\bfw_1,\bfw_2\in\scrW$ satisfying $\inn{\matr{C}_{\lambda_0,s_0,i_0,j_0}\bfw_1, \bfw_2}\ne 0$ with $\norm{\lambda_0}=R$ and $s_0$ maximal possible. We may assume $i_0=j_0=1$.
By \eqref{eq intm: double U-coset} and \eqref{eq intm: expansion of Phi}, we thus have
\begin{equation*}
    \int_G \modinn{\pi(g)\bfw_1,\bfw_2}^q\dif g \asymp \sum_{i,j,n} \mu(Uk_i a^n l_jU) \norm{\sum_{\lambda,s} n^s \lambda^n \inn{\matr{C}_{\lambda,s,i,j}\bfw_1, \bfw_2}}^q.
\end{equation*}
Since $[K:U]<\infty$, we have $\mu(Uk_i a^n l_j U) \asymp \mu(Ka^n K)$ uniformly in $n$.
Following \cite[\S4]{heinig-laat2024group}, let $\delta\coloneq (d_0-1)(d_1-1)$. Then \cite[Prop 4.1]{heinig-laat2024group} gives
\begin{equation*}
    \mu(Ka^n K) = \frac{d_0}{d_0 - 1} \delta^{\frac{\kappa n}{2}},
\end{equation*}
whence uniformly in $n$,
\begin{equation*}
    \mu(U k_i a^n l_j U) \asymp \delta^{\frac{\kappa n}{2}}.
\end{equation*}
By writing $C_{\lambda,s,i,j}\coloneq \inn{\matr{C}_{\lambda,s,i,j}\bfw_1, \bfw_2}$, we get
\begin{equation}\label{eq ppty: expansion of int Lq}
    \int_G \modinn{\pi(g)\bfw_1,\bfw_2}^q\dif g \asymp \sum_{i,j,n} \delta^{\frac{\kappa n}{2}} \norm{\sum_{\lambda,s } C_{\lambda,s,i,j} n^s \lambda^n }^q.
\end{equation}

\begin{lemma}\label{lem: Lq iff delta R < 1}
    $\inn{\pi(\cdot)\bfw_1,\bfw_2}\in\rmL^q(G)$ iff $\delta^{\kappa/2}R^q<1$.
\end{lemma}
\begin{proof}
    If $\delta^{\kappa/2}R^q<1$, then Minkowski's inequality on the RHS of \eqref{eq ppty: expansion of int Lq} gives
    \begin{equation*}
        \int_G \modinn{\pi(g)\bfw_1,\bfw_2}^q\dif g \ll \sum_{n} n^{\dim(\scrV^U)} \delta^{\frac{\kappa n}{2}} R^{nq} < +\infty.
    \end{equation*}

    On the other hand, if $\inn{\pi(\cdot)\bfw_1,\bfw_2}\in\rmL^q(G)$ and we assume by contradiction that $\delta^{\kappa/2}R^q\geq 1$. From \eqref{eq ppty: expansion of int Lq} we get
    \begin{equation}\label{eq intm: lower bound Lq}
        \int_G \modinn{\pi(g)\bfw_1,\bfw_2}^q\dif g \gg \sum_n \left(\delta^{\frac{\kappa}{2}}R^q\right)^n \norm{\sum_{\lambda,s} C_{\lambda,s,1,1} n^s \left(\frac{\lambda}{R}\right)^n }^q.
    \end{equation}
    Recall that we have chosen $\bfw_1,\bfw_2$ so that $C_{\lambda_0,s_0,1,1}\ne 0$ for some $\lambda_0$ of modulus $R$ and $s_0$ maximal possible. We have thus
    \begin{equation*}
        n^{-s_0} \sum_{\lambda, s} C_{\lambda,s,1,1} n^s \left(\frac{\lambda}{R}\right)^n = \sum_{\norm{\lambda}=R} C_{\lambda,s_0,1,1} \left(\frac{\lambda}{R}\right)^n + o(1),
    \end{equation*}
    where the leading term is a nonzero finite sum of powers of unit complex numbers.
    Hence there is some constant $c>0$ such that
    \begin{equation*}
        \norm{\sum_{\lambda,s} C_{\lambda,s,1,1} n^s \left(\frac{\lambda}{R}\right)^n } \geq c n^{s_0}, \quad \text{for infinitely many $n$.}
    \end{equation*}
    But then in \eqref{eq intm: lower bound Lq} the $\rmL^q$-norm diverges, a contradiction. Hence $\delta^{\kappa/2}R^q < 1$.
\end{proof}

Using $\calH(G)$, we can now control the integrability of all $K$-finite coefficients.
\begin{proof}[Proof of \autoref{lem: open condition in p, Lp of K-finite coefficients} for class \ref{ks3}]
    By hypothesis, $\inn{\pi(\cdot)\bfw_1,\bfw_2}$ lies in $\rmL^p$, whence \autoref{lem: Lq iff delta R < 1} yields $\delta^{\kappa/2} R^p < 1$ which remains valid for all $q<p$ sufficiently close to $p$. Again \autoref{lem: Lq iff delta R < 1} implies that $\inn{\pi(\cdot)\bfw_1,\bfw_2}$ lies in $\rmL^q$ for some $q\in (1,p)$, but then \autoref{lem: reduction to k-finite} implies that all $K$-finite coefficients lie in $\rmL^q$.
\end{proof}

\subsection{Conclusion of proof}
Using the results above, we first prove \autoref{mythm: S-algebraic, Lp irreducible ur} for the classes \ref{ks1}, \ref{ks2} and \ref{ks3}. Note that it is trivial for compact groups.
\begin{proof}[Proof of \autoref{mythm: S-algebraic, Lp irreducible ur} for primitive classes]
    Let $\pi\in\wh{G}$ be with a nonzero matrix coefficient in $\rmL^p(G)$. By \autoref{lem: reduction to k-finite}, all $K$-finite matrix coefficients of $\pi$ lie in $\rmL^p(G)$. By \autoref{lem: open condition in p, Lp of K-finite coefficients}, the $K$-finite coefficients lie in $\rmL^{q_0}(G)$ for some $q_0 \in (2,p)$. Hence, $\pi$ is strongly $\rmL^{q_0+}$. Then \autoref{thm samei-wiersma: almost Lp implies uniformly Lp, kunze-stein unimodular} implies that $\pi$ is uniformly $\rmL^{q_0+}$, whence $\pi$ is uniformly $\rmL^q$ for $q\in(q_0, p)$, as desired.
\end{proof}

\subsubsection*{Permanence}
It suffices to verify that the statement of \autoref{mythm: S-algebraic, Lp irreducible ur} is preserved under the elementary operations \ref{a}, \ref{b}, \ref{c}.
The proof below proceeds similarly as that of \autoref{lem: irreducible implies admissible}.
\begin{proof}[Proof of permanence of \autoref{mythm: S-algebraic, Lp irreducible ur}]
    (a) Let $G=G_1\times G_2$ with both $G_i$ satisfying the theorem. Given an irreducible $(\pi, \scrV)\cong (\pi_1,\scrV_1)\boxtimes (\pi_2,\scrV_2)$, we have $\scrV_K = (\scrV_1)_{K_1} \otimes (\scrV_2)_{K_2}$.
    Since $\pi$ has a nonzero coefficient in $\rmL^p$, \autoref{lem: reduction to k-finite} implies that for any $\bfv_i,\bfw_i\in(\scrV_i)_{K_i}$ one has
    \begin{equation*}
        G\ni (g_1, g_2) \mapsto \inn{\pi_1(g_1)\bfv_1,\bfw_1} \inn{\pi_2(g_2)\bfv_2,\bfw_2}  \quad\text{lies in $\rmL^p(G)$}.
    \end{equation*}
    Hence, both $\pi_i$ are strongly $\rmL^p$ for $G_i$. The theorem thus implies that they are both uniformly $\rmL^{q_0}$ for some $q_0\in(2,p)$. Thus $\pi$ is strongly $\rmL^{q_0}$ and hence uniformly $\rmL^{q_0+}$ by \autoref{thm samei-wiersma: almost Lp implies uniformly Lp, kunze-stein unimodular}. Thus $\pi$ is uniformly $\rmL^q$ for $q\in(q_0, p)$.

    (b1) Let $G\lhd G_0$ be a closed finite-index normal subgroup with $G_0$ satisfying the theorem and $K_0=K$. For $\pi\in\wh{G}$ with a nonzero coefficient in $\rmL^p(G)$, consider $\pi_0\coloneq \Ind_G^{G_0}\pi$. 
    \autoref{lem: reduction to k-finite} shows that all $K$-finite matrix coefficients of $\pi$ are in $\rmL^p(G)$, and thus the finite-coset model of $\Ind_G^{G_0}\pi$ implies that all $K$-finite matrix coefficients are in $\rmL^p(G_0)$.
    Now by Mackey, $\pi_0$ is a finite direct sum of irreducible unitary representations of $G_0$.
    Hence every summand of $\pi_0$ has nonzero $K$-finite coefficients in $\rmL^p(G_0)$, to which we can apply the theorem and obtain a common $q\in (2,p)$ so that $\pi_0$ is uniformly $\rmL^q$. Since $\pi_0|_{G}\supset \pi$, we deduce that $\pi$ is uniformly $\rmL^q$.

    (b2) Let $G\rhd G_0$ be a finite-index overgroup with $G_0$ closed normal and satisfying the theorem. Given $\pi\in\wh{G}$ with a nonzero $\rmL^p$ coefficient, all $K$-finite coefficients of $\pi$ are also $\rmL^p$ by \autoref{lem: reduction to k-finite}.
    Since the restriction $\pi|_{G_0}$ is a finite direct sum of irreducible $G_0$-representations by Mackey, we deduce that the $K_0$-finite coefficients of each summand lie in $\rmL^p(G_0)$.
    Applying the theorem on $G_0$ to the summands yields a common $q\in (2,p)$ such that $\pi|_{G_0}$ is uniformly $\rmL^q$.
    However, because of finite index, a matrix coefficient of $\pi$ restricts to a translate of a coefficient of $\pi|_{G_0}$ on each $G_0$-coset, whence $\pi$ is uniformly $\rmL^q$ as well.

    (b3) For the general case $G>G_0$ or $G<G_0$ of finite index, it suffices to pass to the normal core and reduce to (b1) or (b2).

    (c) Let $G=G_0/C$ be a quotient by a compact normal subgroup $C$ with $G_0$ satisfying the theorem. Irreducible representations of $G$ are identified with $C$-trivial irreducible representations of $G_0$. The theorem for $G$ therefore follows easily from that for $G_0$.
\end{proof}

\section{The cyclic question: a counterexample}
    Finally we provide a counterexample to the cyclic version \autoref{conj samei-wiersma: cyclic}.
    The construction exploits the failure of uniform $\rmL^p$-integrability to persist at the endpoint of a convergent family in the complementary series of $\SL(2,\bbR)$.

    Let $G=\SL(2,\bbR)$ and $p\in(2,\infty)$.
    Following \emph{e.g.} \cite[\S9.5]{einsiedler-ward2025unitary}, we parametrize the complementary series $\setdef{\gamma^s}$ of $G$ by $s\in (0,1)$ so that the spherical function associated with $\gamma^s$ lies in $\rmL^{q}(G)$ precisely when $q>p_s\coloneq 2/(1-s)$.
    By \autoref{thm samei-wiersma: almost Lp implies uniformly Lp, kunze-stein unimodular}, each $\gamma^s$ is uniformly $\rmL^p$ for $p>p_s$ and not uniformly $\rmL^p$ for $p\leq p_s$.

    Fix any $s\in (0,1)$. Given a strictly increasing sequence $(s_n)_{n\in\bbN}$ in $(0,s)$ with $s_n\nearrow s$, define
    \begin{equation*}
        (\pi, \scrV) \coloneq \bigoplus_{n\in\bbN} (\gamma^{s_n}, \scrV^{s_n}).
    \end{equation*}

    \begin{proposition}\label{prop: disprove samei-wiersma conj, SL2R}
        Given any $p\in(2,\infty)$, set $s\coloneq 1-2/p$ in the construction above.
        Then $\pi$ admits a cyclic vector $\bfv$ with $\Innerprod{\pi(\cdot)\bfv,\bfv}\in \Lspace^p(G)$, but $\pi$ is not uniformly $\Lspace^p$.
    \end{proposition}
    \begin{proof}
        In the Fell topology, the sequence $\gamma^{s_n}$ converges to $\gamma^s$, so $\gamma^s$ is weakly contained in $\pi$, \emph{cf.} \cite[\S9.6]{einsiedler-ward2025unitary}. Since $\gamma^s$ is not uniformly $\rmL^p$, neither is $\pi$ by \autoref{lem: uniformly Lp passes to weakly contained}.

        Let $\bfv_n$ be a unit vector in $\scrV^{s_n}$ and set $c_n\coloneq \inn{\gamma^{s_n}(\cdot)\bfv_n,\bfv_n}$.
        Since $\gamma^{s_n}$ is uniformly $\rmL^p$, we have $c_n\in\rmL^p(G)$.
        We choose a sequence $(a_n)$ in $(0,1)$ decaying sufficiently fast to 0, so that 
        \begin{equation*}
            \text{(1) $\sum_n a_n^2 <\infty$ \; and \; (2) $\sum_n a_n^{2} \Norm{c_n}_p <\infty$.}
        \end{equation*}

        Now define the vector $\bfv\coloneq\bigoplus_{n} a_n \bfv_n$.
        Then (1) implies that $\bfv\in\scrV$. By Minkowski, (2) implies that the matrix coefficient 
        % $\inn{\pi(g)\bfv, \bfv} = \sum_n a_n^2 c_n(g)$ lies in $\rmL^p(G)$.
        \begin{equation*}
            \inn{\pi(g)\bfv, \bfv} = \sum_n a_n^2 c_n(g) \quad \text{lies in $\rmL^p(G)$.}
        \end{equation*}

        To conclude, we argue that $\bfv$ is cyclic for $\pi$. 
        Since the summands $\gamma^{s_n}$ are mutually inequivalent irreducible, every closed $G$-invariant subspace of $\scrV$ is the Hilbert direct sum of a subfamily of $\setdef{\scrV^{s_n}:n}$.
        Since the projection of $\bfv$ to every $\scrV^{s_n}$ is the nonzero vector $a_n\bfv_n$, the closed cyclic subspace $\scrW$ generated by $\bfv$ has nonzero projection onto every summand. Thus $\scrW$ contains every $\scrV^{s_n}$, whence $\bfv$ is cyclic.
    \end{proof}

    \bibliographystyle{plain}
    {\small \bibliography{reflist.bib}}

\begin{thebibliography}{10}

\bibitem{amann2003groups}
Olivier~Eric Amann.
\newblock {\em Groups of tree-automorphisms and their unitary representations}.
\newblock PhD thesis, ETH Zurich, 2003.

\bibitem{bekka-harpe2020unitary}
Bachir Bekka and Pierre de~la Harpe.
\newblock {\em Unitary {R}epresentations of {G}roups, {D}uals, and
  {C}haracters}, volume 250 of {\em Math. Surv. Monogr.}
\newblock Providence, RI: American Mathematical Society (AMS), 2020.

\bibitem{bernshtein1974all-reductive}
I.~N. Bernshtein.
\newblock All reductive $p$-adic groups are tame.
\newblock {\em Functional Analysis and Its Applications}, 8(2):91--93, April
  1974.

\bibitem{casselman-milicic1982asymptotic}
William Casselman and Dragan Mili{\v c}i{\'c}.
\newblock Asymptotic behavior of matrix coefficients of admissible
  representations.
\newblock {\em Duke Mathematical Journal}, 49(4), December 1982.

\bibitem{cowling1978the-kunze-stein}
Michael Cowling.
\newblock The {K}unze--{S}tein phenomenon.
\newblock {\em Ann. Math.}, 107(2):209--234, 1978.

\bibitem{dixmier1977c-algebras}
Jacques Dixmier.
\newblock {\em {{\(C^*\)}}-algebras. {Translated} by {Francis} {Jellett}},
  volume~15 of {\em North-Holland Math. Libr.}
\newblock Elsevier (North-Holland), Amsterdam, 1977.

\bibitem{einsiedler-ward2025unitary}
Manfred Einsiedler and Thomas Ward.
\newblock {\em Unitary Representations and Unitary Duals}.
\newblock Springer Nature Switzerland, 2025.

\bibitem{gorfine2026a-spectral}
Yuval Gorfine.
\newblock A spectral gap absorption principle.
\newblock {\em Math. Ann.}, 395(2):25, 2026.

\bibitem{gorodnik-nevo2009the-ergodic}
Alexander Gorodnik and Amos Nevo.
\newblock {\em The Ergodic Theory of Lattice Subgroups (AM-172)}.
\newblock Princeton University Press, December 2009.

\bibitem{harish-chandra1953representations}
Harish-Chandra.
\newblock Representations of a semisimple {Lie} group on a {Banach} space. {I}.
\newblock {\em Trans. Am. Math. Soc.}, 75:185--243, 1953.

\bibitem{harish-chandra1959some}
Harish-Chandra.
\newblock Some results on differential equations and their applications.
\newblock {\em Proc. Natl. Acad. Sci. USA}, 45:1763--1764, 1959.

\bibitem{harish-chandra1970harmonic}
Harish-Chandra.
\newblock {\em Harmonic Analysis on Reductive p-adic Groups}.
\newblock Springer Berlin Heidelberg, 1970.

\bibitem{harish-chandra1984differential}
Harish-Chandra.
\newblock Differential equations and semisimple {Lie} groups.
\newblock In V.~S. Varadarajan, editor, {\em Harish-Chandra Collected Papers:
  Volume III (1959--1968)}, pages 57--120. Springer-Verlag, New York, 1984.
\newblock Unpublished manuscript [1960b].

\bibitem{harish-chandra1984some}
Harish-Chandra.
\newblock Some results on differential equations.
\newblock In V.~S. Varadarajan, editor, {\em Harish-Chandra Collected Papers:
  Volume III (1959--1968)}, pages 7--48. Springer-Verlag, New York, 1984.
\newblock Unpublished manuscript [1960a].

\bibitem{heinig-laat2024group}
Dennis Heinig, Tim de~Laat, and Timo Siebenand.
\newblock Group {{\(C^*\)}}-algebras of locally compact groups acting on trees.
\newblock {\em Int. Math. Res. Not.}, 2024(10):8520--8539, 2024.

\bibitem{knapp1986representation}
Anthony~W. Knapp.
\newblock {\em Representation Theory of Semisimple Groups. {An} overview based
  on examples}, volume~36 of {\em Princeton Math. Ser.}
\newblock Princeton University Press, Princeton, NJ, 1986.

\bibitem{kunze-stein1960uniformly}
R.~A. Kunze and E.~M. Stein.
\newblock Uniformly bounded representations and harmonic analysis of the $2
  \times 2$ real unimodular group.
\newblock {\em Am. J. Math.}, 82(1):1--62, 1960.

\bibitem{nebbia1988groups}
Claudio Nebbia.
\newblock Groups of isometries of a tree and the {Kunze-Stein} phenomenon.
\newblock {\em Pacific Journal of Mathematics}, 133(1):141--149, May 1988.

\bibitem{nebbia1999groups}
Claudio Nebbia.
\newblock Groups of isometries of a tree and the {CCR} property.
\newblock {\em Rocky Mountain Journal of Mathematics}, 29(1), March 1999.

\bibitem{samei-wiersma2024exotic}
Ebrahim Samei and Matthew Wiersma.
\newblock Exotic {$C^*$}-algebras of geometric groups.
\newblock {\em J. Funct. Anal.}, 286(2):110228, January 2024.

\bibitem{semal2023irreducibly}
Lancelot Semal.
\newblock {\em Irreducibly represented Lie groups and Nebbia's CCR conjecture
  on trees}.
\newblock PhD thesis, Universit{\'e} Catholique de Louvain, 2023.
\newblock arXiv:2306.04310.

\bibitem{semal2024radu}
Lancelot Semal.
\newblock {R}adu groups acting on trees are {CCR}.
\newblock {\em Journal of the Australian Mathematical Society},
  117(2):149--186, March 2024.

\bibitem{silberger1982asymptotics}
Allan~J. Silberger.
\newblock Asymptotics and integrability properties for matrix coefficients of
  admissible representations of reductive $p$-adic groups.
\newblock {\em Journal of Functional Analysis}, 45(3):391--402, February 1982.

\bibitem{veca2002the-kunze-stein}
Alessandro Veca.
\newblock {\em The {K}unze--{S}tein phenomenon}.
\newblock PhD thesis, Univ. of New South Wales, 2002.

\bibitem{willis1994the-structure}
G.~Willis.
\newblock The structure of totally disconnected, locally compact groups.
\newblock {\em Math. Ann.}, 300(2):341--363, 1994.

\end{thebibliography}

    \vspace*{1em}
    \noindent{\scshape Siwei Liang}: LMO, Université Paris-Saclay, Orsay, France\\
    Email: \texttt{siwei.liang@universite-paris-saclay.fr}
\end{document}